\documentclass[a4paper,reqno]{amsart}
\usepackage{amssymb, amsmath, amscd}
\usepackage{color}
\def\black{\color{black}}

\def\id{\operatorname{Id}}
\newtheorem{theorem}{Theorem}[section]
\newtheorem{proposition}{Proposition}[section]
\newtheorem{definition}{Definition}[section]
\newtheorem{lemma}{Lemma}[section]
\newtheorem{remark}{Remark}[section]
\makeatletter
 \@addtoreset{equation}{section}
\makeatother
\begin{document}
\title{Transplanting geometrical structures}
\author{Y. Euh, P. Gilkey, J.H. Park, and K. Sekigawa}
\address{YE and JHP: Department of Mathematics, Sungkyunkwan University, Suwon
440-746, Korea \\ and the Korea Institute for Advanced Study, Seoul
130-722, Korea} \email{prettyfish@skku.edu and parkj@skku.edu}
\address{PG: Mathematics Department, University of Oregon, Eugene OR 97403 USA}
\email{gilkey@uoregon.edu}
\address{KS: Department of Mathematics, Niigata University, Niigata, Japan.}
\email{sekigawa@math.sc.niigata-u.ac.jp}
\begin{abstract}{We say that a germ $\mathcal{G}$ of a geometric structure can be
transplanted into a manifold $M$ if there is a suitable geometric structure on $M$
which agrees with $\mathcal{G}$ on a neighborhood of some point $P$ of $M$. We show for
a wide variety of geometric structures that this transplantation is always possible provided
that $M$ does in fact admit some such structure of this type. We use this result to show
that a curvature identity which holds in the category of compact
manifolds admitting such a structure holds for germs as well and
we present examples illustrating this result. We also use this result to show geometrical realization
problems which can be solved for germs of structures can in fact be solved in the compact setting as well.
}\end{abstract}

\maketitle
\section{Introduction}
We shall consider the following geometric structures; precise definitions will be given in Section~\ref{sect-2}.
We fix the dimension $m$ of the underlying manifold and also the signature $(p,q)$ where relevant.

\begin{definition}\label{defn-1.1}
\rm Consider the
possible structures:\begin{enumerate}
\item Affine structures.
\item Pseudo-Riemannian structures of suitable signature.\item Almost (para)-complex structures in dimension $m=2\bar m$.
\item Almost (para)-Hermitian structures of suitable signature.
\item (Para)-complex  structures in dimension $m=2\bar m$.
\item (Para)-Hermitian structures of suitable signature.
\item (Para)-K\"ahler structures of suitable signature.
\item Weyl structures of suitable signature.

\item (Para)-K\"ahler--Weyl structures of suitable signature.
\end{enumerate}\end{definition}

Let $\mathcal{S}$ be a structure from this list, let $\mathcal{G}\in\mathcal{S}$ be the germ of such a structure, and
let $M$ be a smooth manifold which admits a structure $S_M\in\mathcal{S}$.
Fix a point $P\in M$. We say that $\mathcal{G}$ can be {\it transplanted} in $M$
if there exists a structure $\tilde S_M\in\mathcal{S}$
such that $\tilde S_M$ near $P$ is locally isomorphic to $\mathcal{G}$
and $\tilde S_M$ agrees with $S_M$ away from $P$. We shall make this
more precise in Section~\ref{sect-2}-- it is necessary
to assume that $M$ admits such a structure
to avoid topological difficulties. For example,
if $\mathcal{S}$ denotes the structure of Lorentzian metrics of
signature $(1,m-1)$, then not every manifold will admit such a structure; similarly, if we are working with
almost complex structures, not every manifold admits an almost complex structure.

Such problems arise in many contexts.  One can often establish universal curvature identities
for compact manifolds by considering the Euler-Lagrange equations of certain characteristic
classes -- see, for example, results in \cite{EPSx, EPS2,EPS3,EPS4,GPS,LPS}. One then wants to show these identities hold more generally
and this involves a transplanting problem. Similarly, when establishing geometric realization results,
one often constructs examples which are only defined on
a small neighborhood of the origin -- the discussion in \cite{BGN12} provides a nice summary
of these problems and we refer to other results in \cite{BGM10, BGMR10,GN12, GNS11}.  And one
wants to then deduce these geometric realization results also hold in the compact setting.

\subsection{Transplantation} The following is the first main result of this paper; a more precise statement for each of the structures
in Definition~\ref{defn-1.1} will be given subsequently in Section~\ref{sect-2}.

\begin{theorem}\label{thm-1.1}
All the structures of Definition~\ref{defn-1.1} can be transplanted.\end{theorem}

\subsection{Curvature identities}
Theorem~\ref{thm-1.1} then yields the second main result of this paper which
motivated our investigations in the first instance; we will present several applications
 in Section~\ref{sect-3}:

\begin{theorem}\label{thm-1.2}
Let $\mathcal{S}$ be a structure of Definition~\ref{defn-1.1}.
A curvature identity which holds for every $\xi\in\mathcal{S}(M)$
for $M$  compact necessarily holds without the assumption of compactness.\end{theorem}

\subsection{Geometric realizability}
Theorem~\ref{thm-1.1} and results described in
\cite{BGN12} also yield the following result which formed part of the motivation of our paper; we will
present several examples illustrating this result in Section~\ref{sect-4}:

\begin{theorem}\label{thm-1.3}
Every algebraic model of the curvature tensor of a structure from Definition~\ref{defn-1.1} is geometrically realizable by a compact
manifold.
\end{theorem}

\section{The proof of Theorem~\ref{thm-1.1}}\label{sect-2}

\subsection*{Notational conventions} We introduce some basic notational conventions that we shall employ throughout Section~\ref{sect-2}.

\begin{definition}\label{defn-2.1}
\rm Let $B_{3r}$ be the ball of radius $3r$ about the origin in $\mathbb{R}^m$ and let $h$
be a smooth $k$-tensor defined on $B_{3r}$. Let $||h||_{3r}$ be the $C^0$ norm of $h$  and let
$||h||_{3r}^1$ be the $C^1$ norm of
$h$ on $B_{3r}$, i.e.
\begin{eqnarray*}
&&||h||_{3r}:=\sup_{x\in B_{3r},1\le i_1\le m,...,1\le i_k\le m}|h_{i_1...i_k}(x)|,\\
&&||h||_{3r}^1:=||h||_{3r}+\sup_{x\in B_{3r},1\le j\le i_1\le m,...,1\le i_k\le m}
|\partial_{x_j}h_{i_1...i_k}(x)|\,.
\end{eqnarray*}
We shall also denote these norms by $||h||$ and $||h||^1$ when no confusion is likely to result.
We shall need $||\cdot||^1$ in
Theorem~\ref{thm-2.2} and Theorem~\ref{thm-2.7} to be able to study geodesic completeness in Section~\ref{sect-2.10}; in
the remaining results we will either study $C^0$ approximating transplants or simple transplants without estimates.\end{definition}

\begin{definition}\label{defn-2.2}
\rm Let $P_i$ be the base points of manifolds $M_i$. Let
$\mathcal{S}$ be a structure from the list given in
Definition~\ref{defn-1.1}. Let $S_1\in\mathcal{S}(M_1)$ be the germ
of a suitable structure on $M_1$ at $P_1$ and let
$S_2\in\mathcal{S}(M_2)$ be  a structure which is defined on all of
$M_2$. We will choose suitably normalized local coordinate systems
$\vec x=(x^1,...,x^m)$  centered at $P_1$ which are defined on an
open set $\mathcal{U}_1\subset M_1$ and suitably normalized
coordinates $\vec y=(y^1,...,y^m)$ centered at $P_2$ which are
defined on a open set $\mathcal{U}_2\subset M_2$. We use these
coordinates to identify $\mathcal{U}_1$ and $\mathcal{U}_2$ with
$B_{3r}$ and $P_1=P_2=0$ for some $r>0$; we will often shrink $r$ in
the course of a particular discussion. We use the identification
$\mathcal{U}_1=B_{3r}$ to regard $S_1$ as defining a structure
on $B_{3r}$ and we use the identification $\mathcal{U}_2=B_{3r}$ to
regard $S_2$ as defining a structure on $B_{3r}$ as well. Our task
will be to find a structure $\tilde S_2\in\mathcal{S}(B_{3r})$ so
that $\tilde S_2=S_1$ on $B_r$ and so that $\tilde S_2=S_2$ on
$B_{2r}^c$; $\tilde S_2$ can then be extended to all of $M_2$ to
agree with $S_2$ on $\mathcal{U}_2^c$. In this setting, we will say
that {\it $\tilde S_2$ is isomorphic to $S_1$ near $P_2$ and agrees
with $S_2$ away from $P_2$}.
\end{definition}

\begin{definition}\label{defn-2.3}
\rm Let $\phi\in C_0^\infty(B_{3})$ be a mesa function so
$\phi\equiv1\text{ on }B_1$, so $\phi\equiv0$ on $B_{2}^c$, and so $0\le\phi\le 1$. We let
$\phi_r:=\phi(x/r)$ be the corresponding mesa function on $B_{3r}$.
There is a constant $C$ which is independent of $r$ so that
\begin{equation}\label{eqn-2.a}
||\phi_r||_{3r}\le 1\text{ and }||\phi_r||_{3r}^1\le Cr^{-1}\,.
\end{equation}\end{definition}

\begin{definition}\label{defn-2.4}
\rm We may decompose $\otimes^2T^*M=S^2(M)\oplus\Lambda^2(M)$ as the direct
sum of the symmetric $2$-forms $S^2 (M)$ and the alternating $2$-forms $\Lambda^2(M)$.
 If $J$ is an endomorphism of $TM$ with $J^2=\pm\id$, let
$S_\pm^2(M)\subset S^2(M)$ and $\Lambda_\pm^2(M)\subset\Lambda^2(M)$ be the $\pm1$ eigenbundles of the
associated action by $J$ on $S^2(M)$ and on $\Lambda^2(M)$.
If $\mathbb{V}$ is a smooth vector bundle over $M$, let $C^\infty(\mathbb{V})$
denote the space of smooth sections to $\mathbb{V}$.\end{definition}

We now treat the structures of Definition~\ref{defn-1.1} seriatim; we have chosen the order
of presentation to be able to use results about certain structures in the study of subsequent structures.
We refer to \cite{CFG96}
for further details concerning para-complex geometry. Although partitions of
unity are used in the proofs, certain structures require constructions which are far from direct.
We shall employ{\it (para)-K\"ahler potentials} to examine (para)-K\"ahler geometry. For use as input in studying
subsequent more delicate structures, it will be convenient in many instances to be able to choose $\tilde S_2$ so that
$||\tilde S_2-S_2||<\epsilon$ or so that $||\tilde S_2-S_2||^1<\epsilon$
when $\epsilon>0$ is given a-priori.

\subsection{Affine structures}\label{sect-2.1}
The pair $(M,\nabla)$ is called an {\it affine manifold} and $\nabla$ is said to be an {\it affine connection}
if $\nabla$ is a torsion free connection on the tangent bundle $TM$.

\begin{theorem}\label{thm-2.1} Let $\epsilon>0$ be given. Let $\nabla_1$ be the germ of a
torsion free connection at $P_1\in M_1$ and let $\nabla_2$ be a torsion free connection on $M_2$. There exists
a torsion free connection $\tilde\nabla_2$ on $M_2$ so that $\tilde\nabla_2$ is equal to $\nabla_2$ away from $P_2$,
so that $\tilde\nabla_2$ is isomorphic to $\nabla_1$ near $P_2$, and so that $||\tilde\nabla_2-\nabla_2||<\epsilon$.
\end{theorem}

\begin{proof} Let
$\nabla_{1,\partial_{x_i}}\partial_{x_j}=\Gamma_{1,ij}{}^k\partial_{x_k}
\text{ and
}\nabla_{2,\partial_{y_i}}\partial_{y_j}=\Gamma_{2,ij}{}^k\partial_{y_k}$
define the {\it Christoffel symbols} of the connections where we
adopt the {\it Einstein convention} and sum over repeated indices.
Since $\nabla_1$ and $\nabla_2$ are torsion free, we can normalize
the coordinate system $\vec x$ (resp. $\vec y$) so that $\Gamma_1(P_1)=0$
(resp. $\Gamma_2(P_2)=0)$; see, for
example, the discussion in \cite{BGN12}. Having chosen $\vec x$ and
$\vec y$, we now identify $P_1=P_2=0$ and
$\mathcal{U}_1=\mathcal{U}_2=B_{3r}$. It is crucial here and it will
be crucial in subsequent arguments to normalize the structures
before identifying $P_1$ with $P_2$; the freedom to adjust $\vec x$
and $\vec y$ separately is an important one. The difference $\Phi$
of two connections is tensorial; we set
$$\Phi_{ij}{}^k:=\Gamma_{1,ij}{}^k-\Gamma_{2,ij}{}^k\text{ on }B_{3r}\,.$$
Since $\Gamma_1(0)=\Gamma_2(0)=0$, $\Phi(0)=0$ so by shrinking $r$ if necessary,
we may assume that $||\Phi||_{3r}<\varepsilon$.
Set $\tilde\nabla_2=\phi_r\cdot\nabla_1+(1-\phi_r)\cdot\nabla_2$.
Then $\tilde\nabla_2$ is torsion free, $\tilde\nabla_2$ agrees
with $\nabla_1$ on $B_{r}$, $\tilde\nabla_2$ agrees with $\nabla_2$ on $B_{2r}^c$, and
\medbreak\qquad\qquad\qquad
$||\tilde\nabla_2-\nabla_2||_{3r}=||\phi_r\nabla_1+(1-\phi_r)\nabla_2-\nabla_2||_{3r}=||\phi_r\Phi||_{3r}<\epsilon$.
\end{proof}

\subsection{Pseudo-Riemannian manifolds}\label{sect-2.2}
A {\it pseudo-Riemannian metric of signature $(p,q)$} is a non-degenerate bilinear form
$g\in C^\infty(S^2(T^*M))$ of signature $(p,q)$
on $TM$; the pair $(M,g)$ is a {\it pseudo-Riemannian manifold of signature $(p,q)$}. Not every manifold admits such a metric; for example the even dimensional spheres do not admit pseudo-Riemannian metrics of signature $(p,q)$ if $p>0$ and if $q>0$.
\begin{theorem}\label{thm-2.2} Let $\epsilon>0$ be given. Let $g_1$ be the germ of a
pseudo-Riemannian metric of signature $(p,q)$ at $P_1\in M_1$
 and let $g_2$ be a pseudo-Riemannian metric of signature $(p,q)$ on $M_2$. There exists
a pseudo-Riemannian metric $\tilde g_2$ on $M_2$ so that $\tilde g_2$ is equal to $g_2$ away from $P_2$, so that $\tilde g_2$ is isomorphic to $g_1$ near $P_2$, and so that $||g_2-\tilde g_2||^1<\epsilon$.
\end{theorem}

\begin{proof}
Choose local coordinates $\vec x$ (resp. $\vec y$) near $P_1\in M_1$ (resp. $P_2\in M_2$).
Let
$$g_{1,ij}:=g_1(\partial_{x_i},\partial_{x_j})\text{ and }g_{2,ij}:=g_2(\partial_{y_i},\partial_{y_j})\,.$$
By choosing geodesic polar coordinates, we may assume the first derivatives of the metric vanish at
the points $P_1$ and $P_2$, i.e.
$$\partial_{x_k}g_{1,ij}(P_1)=0\text{ and }\partial_{y_k}g_{2,ij}(P_2)=0\text{ for }1\le i,j,k\le m\,.$$
Since $g_1$ and $g_2$ have signature $(p,q)$, by making a linear change of coordinates
we can further normalize the choice of $\vec x$ and $\vec y$ so:
$$
g_{1,ij}(P_1)=g_{2,ij}(P_2)=
   \left\{\begin{array}{rl}0&\text{if }i\ne j\\-1&\text{if }i=j\le p\\+1&\text{if }i=j>p\end{array}\right\}\,.
$$
We now identify $P_1=P_2=0\in\mathbb{R}^m$. Because  the first derivatives of the metric vanish at $0$ and because
$g_1(0)=g_2(0)$, we have
\begin{equation}\label{eqn-2.b}
||g_1-g_2||_{3r}\le Cr^2\text{ and }||g_1-g_2||_{3r}^1\le Cr
\end{equation}
for some constant $C$ which is independent of $r$ if $0\le r<r_0$ is sufficiently small.
We set $\tilde g_2:=\phi_rg_1+(1-\phi_r)g_2\in S^2(T^*M)$. Then
$\tilde g_2=g_1$ on $B_r$, and $\tilde g_2=g_2$ on $B_{2r}^c$. Since $\tilde g-g_2=\phi_r(g_1-g_2)$, we use
Equation~(\ref{eqn-2.a}) and Equation~(\ref{eqn-2.b}) to see:
\begin{eqnarray*}
||(\tilde g_2-g_2)||_{3r}&\le&||\phi_r||_{3r}\cdot||g_1-g_2||_{3r}\le Cr^2,\\
||\tilde g_2-g_2||_{3r}^1&\le&C\{||\phi_r||_{3r}^1\cdot||(\tilde g_2-g_2)||_{3r}
+||\phi_r||_{3r}\cdot||(\tilde g_2-g_2)||_{3r}^1\}\\
&\le&C\cdot r^{-1}\cdot r^2+C\cdot r\le 2C\cdot r
\end{eqnarray*}
where $C$ is a generic constant.
The desired estimate that $||\tilde g_2-g_2||^1<\epsilon$ now follows by choosing $r$ small enough.
Since the set of pseudo-Riemannian metrics on $M_2$ of signature $(p,q)$ is an open subset of $C^\infty(S^2(T^*M))$ in the $C^0$
topology, it follows that $\tilde g_2$ is non-degenerate if $\epsilon$ is chosen sufficiently small.
\end{proof}

\subsection{Almost (para)-complex structures}\label{sect-2.3}
We say that an endomorphism $J_+$ of $TM$
 is an {\it almost para-complex structure}
if $J_+^2=\id$ and if $\operatorname{Tr}(J_+)=0$;
similarly an endomorphism $J_-$ of $TM$ is an {\it almost complex structure}
if $J_-^2=-\id$; it is then immediate that $\operatorname{Tr}(J_-)=0$.
It is convenient to have a common notation
although we shall never be discussing both structures at the same time. The existence
of an almost  (para)-complex structure necessarily
implies $m=2\bar m$ is even. Not every manifold admits such a structures;
the even dimensional spheres $S^m$ for $m\ge8$
do not admit such structures.
The pair $(M,J_\pm)$ is said to be an {\it almost (para)-complex manifold}.

\begin{theorem}\label{thm-2.3} Let $\epsilon>0$ be given. Let $J_{1,\pm}$ be the germ of an
almost (para)-complex structure near $P_1\in M_1$ and let $J_{2,\pm}$ be an almost
(para)-complex structure on $M_2$. There exists
an almost (para)-complex structure $\tilde J_{2,\pm}$ on $M_2$ so that $\tilde J_{2,\pm}$ is equal to $J_{2,\pm}$ away from $P_2$,
so that $\tilde J_{2,\pm}$ is isomorphic to $J_{1,\pm}$ near $P_2$, and so that $||\tilde J_{2,\pm}-J_{2,\pm}||<\epsilon$.
\end{theorem}

\begin{proof}
Choose local coordinates $\vec x$ (resp. $\vec y$) near $P_1\in M_1$ (resp. $P_2\in M_2$).
Make a linear change of coordinates to normalize
the coordinate systems so that the structures are isomorphic on $T_{P_i}(M_i)$:
\begin{equation}\label{eqn-2.c}
\begin{array}{l}
J_\pm(P_1)\partial_{x_i}=\partial_{x_{i+\bar
m}}\quad\text{and}
\quad J_\pm(P_1)\partial_{x_{i+\bar m}}=\pm\partial_{x_i}\quad\text{for}\quad 1\le i\le\bar m,\\
J_\pm(P_2)\partial_{y_i}=\partial_{y_{i+\bar
m}}\quad\text{and}
\quad J_\pm(P_2)\partial_{y_{i+\bar m}}=\pm\partial_{y_i}\quad\text{for}\quad 1\le i\le\bar m.\\
\end{array}
\end{equation}
We now identify $P_1=P_2=0\in\mathbb{R}^m$ and have
$J_{1,\pm}(P_1)=J_{2,\pm}(P_2)$. By shrinking $r$ if
necessary, we may assume
$$||J_{1,\pm}-J_{2,\pm}||_{3r}<\epsilon\,.$$

The convex combination of (para)-complex structures is, of course,
not a (para)-complex structure so we must proceed a bit
differently and not simply examine $\phi_r J_{1,\pm}+(1-\phi_r)J_{2,\pm}$.
For $1\le i\le\bar m$, we define tangent vector fields $e_i$ and $f_i$ on $B_{3r}$ by setting:
$$
e_i:=\partial_{{{x}}_i},\quad e_{i+\bar m}:=J_{1,\pm}e_i, \quad
f_i:=\partial_{y_i},\quad\ f_{i+\bar m}:=J_{2,\pm}f_i\,.
$$
Then $\vec e:=(e_1,...,e_m)$ and $\vec f:=(f_1,...,f_m)$ are linearly independent at $0$ and hence are frames
for $TB_{3r}$ if we shrink $r$. Furthermore, as $\vec e(0)=\vec f(0)$,
 we may assume $||\vec e-\vec f||_{3r}<\epsilon$ if we shrink $r$ still further. We may express
$$\vec f(x)=\Theta(x)\cdot \vec e(x)\text{ on }B_{3r}
\text{ for
}\Theta:B_{3r}\rightarrow\operatorname{GL}(m,\mathbb{R})\text{ with
} \Theta(0)=\id\,.$$ Set $\vec g(x):=\Theta(\phi_r(x)x)\vec
e(x)$. Then $\vec g=\vec f$ on $B_{r}$ and $\vec g=\vec e$
on $B_{2r}^c$. As $\Theta(0)-\id=0$, by shrinking $r$ if
necessary, we may estimate
$$||\vec g-\vec e||_{3r}\le C_m||\Theta-\id||_{3r}\cdot||\vec e||_{3r}<\epsilon$$
for some universal constant $C_m$ only depending on the underlying dimension
of the manifold.
The desired extension $\tilde J_{2,\pm}$ may
then be defined by setting:
\medbreak
\qquad\qquad\qquad$\tilde J_{2,\pm}g_i=g_{i+\bar m}$ and
$\tilde J_{2,\pm}g_{i+\bar m}=\pm g_i$ for $1\le i\le\bar m$.
\end{proof}

\subsection{Almost (para)-Hermitian structures}\label{sect-2.4}
An {\it almost (para)-Hermitian structure} on a manifold $M$ is a pair
$(g,J_\pm)$ where $g$ is a pseudo-Riemannian metric on $M$, where
$J_\pm$ is an almost (para)-complex structure on $M$, and
where $J_\pm^*g=\mp g$. The topological restrictions discussed
previously show not every manifold admits such a structure.
The triple $(M,g,J_\pm)$ is said to be an {\it almost (para)-Hermitian manifold}.

\begin{theorem}\label{thm-2.4} Let $\epsilon>0$ be given. Let $(g_1,J_{1,\pm})$ be the germ of an
almost (para)-Hermitian structure of signature $(p,q)$ near $P_1\in M_1$
and let $(g_2,J_{2,\pm})$ be an almost
(para)-Hermitian structure on $M_2$ of signature $(p,q)$. There exists
an almost (para)-Hermitian structure $(\tilde g_2,\tilde J_{2,\pm})$
on $M_2$ so $(\tilde g_2,\tilde J_{2,\pm})$ is equal to
$(g_2,J_{2,\pm})$ away from $P_2$,
so $(\tilde g_2,\tilde J_{2,\pm})$ is isomorphic to $(g_1,J_{1,\pm})$ near $P_2$,  so
$||\tilde g_2-g_2||<\epsilon$, and so $||\tilde J_{2,\pm}-J_{2,\pm}||<\epsilon$.
\end{theorem}

\begin{proof} We choose normalized coordinate systems $\vec x$ and $\vec y$ so that the relations of Equation~(\ref{eqn-2.c}) hold.
In the complex setting, we have $p=2\bar q$ and $q=2\bar q$ are both even.
Since $g_1$ and $g_2$ have the same signature, we can make a linear change of coordinates preserving the normalizations
of Equation~(\ref{eqn-2.c}) so that in the complex setting,
\begin{equation}\label{eqn-2.d}
g_{\nu,ij}(P_\nu)=\left\{\begin{array}{rllll}0&\text{if}&i\ne j\\
-1&\text{if}&i=j\le\bar p&\text{or}&\bar m+1\le i\le\bar m+\bar p\\
+1&\text{if}&\bar p<i=j\le\bar m&\text{or}&\bar m+1+\bar p<i\le m
\end{array}\right\}.\end{equation}
In the para-complex setting, we have $p=q=\bar m$ since $g_\nu$ necessarily has neutral signature. We normalize the choice of coordinates
so
\begin{equation}\label{eqn-2.e}
g_{\nu,ij}(P_\nu)=\left\{\begin{array}{rllll}0&\text{if}&i\ne j\\
-1&\text{if}&1\le i=j\le\bar m\\
+1&\text{if}&\bar m<i=j\le m
\end{array}\right\}.\end{equation}
We now identify $P_1=P_2=0\in\mathbb{R}^m$ and have $g_1(0)=g_2(0)$ and $J_{1,\pm}(0)=J_{2,\pm}(0)$.
Let $\epsilon>0$ be given. Let $\delta=\delta(\epsilon)$ to be chosen presently.
We use Theorem~\ref{thm-2.2} and Theorem~\ref{thm-2.3} to define an almost (para)-complex structure $\tilde J_{2,\pm}$
and a pseudo-Riemannian metric $g_3$ so that:
$$\begin{array}{rrr}
\tilde J_{2,\pm}=J_{1,\pm}\text{ on }B_{r},& \tilde J_{2,\pm}=J_{2,\pm}\text{ on }B_{2r}^c,&
    ||\tilde J_{2,\pm}-J_{2,\pm}||<\delta,\\
g_3=g_1\text{ on }B_{r},& g_3=g_2\text{ on }B_{2r}^c,&
    ||g_3-g_2||<\delta.
\end{array}$$
We average over the action of $\tilde J_{2,\pm}$ to define:
$\tilde g_2:=\textstyle\frac12\{g_3\mp\tilde J_{2,\pm}^*g_3\}$.
Since $g_1$ and $g_2$ are (para)-Hermitian metrics, we still have that
$$\tilde g_2=g_1\text{ on }B_{r}\text{ and }\tilde g_2=g_2\text{ on }B_{2r}^c\,.$$
Since $\tilde J_{2,\pm}$ is close to $J_{2,\pm}$, since $\tilde g_3$ is close to $g_2$,
and since $J_{2,\pm}^*g_2=\pm g_2$, we have that $||\tilde g_2-g_2||<\epsilon$
for $\delta=\delta(\epsilon)$ suitably chosen.
Since we can always assume $\epsilon$ small, this
implies that $\tilde g_2$ is non-degenerate. \end{proof}

\subsection{Complex and para-complex structures}\label{sect-2.5}
Let $J_\pm$ be an almost (para)-complex structure on $M$. We say that
$J_\pm$ is {\it integrable} and that $(M,J_\pm)$ is a {\it(para)-complex manifold}
if there are coordinates centered at every point of $M$
so that the relations of Equation~(\ref{eqn-2.c}) hold at all points of the chart:
\begin{equation}\label{eqn-2.f}
J_\pm\partial_{x_i}=\partial_{x_{i+\bar m}}\quad\text{and}
\quad J_\pm\partial_{x_{i+\bar m}}=\pm\partial_{x_i}\quad\text{for}\quad 1\le i\le\bar m\,.
\end{equation}
Equivalently, $J_\pm$ is integrable if the {\it Nijenhuis tensor}\index{Nijenhuis tensor} $N_{J_\pm}$ vanishes where:
$$
N_{J_\pm}(X,Y):=[X,Y]\mp J_\pm[J_\pm X,Y]\mp J_\pm [X,J_\pm Y]\pm[J_\pm X,J_\pm Y]\,.
$$
Complex structures and para-complex structures are rigid.
The following result is now immediate - there is no need to perturb the structures:
\begin{theorem}\label{thm-2.5} Let $J_{1,\pm}$ be the germ of a (para)-complex structure near $P_1\in M_1$ and let $J_{2,\pm}$ be a
(para)-complex structure on $M_2$. Then $J_{2,\pm}$ is isomorphic to $J_{1,\pm}$ near $P_2$.
\end{theorem}

\subsection{Hermitian and para-Hermitian structures}\label{sect-2.6}
A {\it (para)-Hermitian structure} is a pair $(g,J_\pm)$ where $g$ is a pseudo-Riemannian metric on $M$, where
$J_\pm$ is an integrable (para)-complex structure on $M$, and where $J_\pm^*g=\mp g$. The triple
$(M,g,J_\pm)$ is said to be a {\it (para)-Hermitian manifold}; we emphasize that $g$ can be indefinite in the complex setting --
necessarily $g$ has neutral signature $(\bar m,\bar m)$ in the para-complex setting.
The arguments given to prove Theorem~\ref{thm-2.4}
now extend to establish the following result; it is not necessary to perturb the (para)-complex structure.

\begin{theorem}\label{thm-2.6}
Let $\epsilon>0$ be given. Let $(M_1,g_1,J_{1,\pm})$ be the germ of a
(para)-Hermitian structure near $P_1\in M_1$ of signature $(p,q)$
and let $(M_2,g_2,J_{2,\pm})$ be a
(para)-Hermitian structure on $M_2$ of signature $(p,q)$.
There exists a pseudo-Riemannian metric $\tilde g_2$ on $M_2$ so that
$(\tilde g_2, J_{2,\pm})$ is a (para)-Hermitian structure on $M_2$ which
 is isomorphic to $(g_1,J_{1,\pm})$ near $P_2$,
so that $\tilde g_2=g_2$ away from $P_2$, and so that
$||\tilde g_2-g_2||<\epsilon$.
\end{theorem}

\subsection{(Para)-K\"ahler geometry}\label{sect-2.7}
Let $(M,g,J_\pm)$ be an almost (para)-Hermitian manifold of signature $(p,q)$. Define the
vector bundles $\Lambda_\mp^2(M)$ and $S_\mp^2(M)$ and define
the {\it (para)-K\"ahler form} $\Omega_\pm\in C^\infty(\Lambda_\mp^2(M))$
by setting:
\begin{equation}\label{eqn-2.g}
\begin{array}{l}
S_\mp^2(M):=\{h\in S^2(M):J_\pm^* h=\mp h\},\\
\Lambda_\mp^2(M):=\{\omega\in\Lambda^2(M):J_\pm^*\omega=\mp\omega\},\vphantom{\vrule height 12pt}\\
\Omega_\pm(x,y):=g(x,J_\pm y)\in C^\infty(\Lambda_\mp^2(M))\,.\vphantom{\vrule height 12pt}
\end{array}\end{equation}
Let $\nabla$ be the Levi-Civita connection of $g$.
One has the following useful result -- see, for example, the discussion in \cite{BGN12,KN63}.
\begin{proposition}\label{prop-2.1}
Let $\nabla^g$ be the Levi-Civita connection of an almost (para)-Herm\-itian
manifold $(M,g,J_\pm)$. The following conditions are equivalent
and if any is satisfied, then $(M,g,J_\pm)$ is said to be a {\rm(para)-K\"ahler manifold} and $(g,J_\pm)$ is said
to be a {\rm (para)-K\"ahler} structure.
\begin{enumerate}
\item  $\nabla^g\Omega_\pm=0$.
\item $\nabla^g J_\pm=0$.
\item $J_\pm$ is integrable and $d\Omega_\pm=0$.
\item For every point $P$ of $M$, there exists a coordinate system centered at $P$ so that
 Equation~{\rm(\ref{eqn-2.f})} is satisfied, so that
Equation~{\rm(\ref{eqn-2.d})} is satisfied in the complex
setting and Equation~{\rm(\ref{eqn-2.e})} is satisfied in the para-complex setting, and so that
the first derivatives of the metric vanish at $P$.\end{enumerate}
\end{proposition}

\begin{theorem}\label{thm-2.7}
Let $\epsilon>0$ be given. Let $(g_1,J_{1,\pm})$ be the germ of a (para)-K\"ahler
structure near $P_1\in M_1$ of signature $(p,q)$ and let $(g_2,J_{2,\pm})$ be a
(para)-K\"ahler structure of signature $(p,q)$ on $M_2$.
There exists a pseudo-Riemannian metric $\tilde g_2$ on $M_2$ so that
$(\tilde g_2,J_{2,\pm})$ is a K\"ahler structure on $M_2$, so that
$(\tilde g_2, J_{2,\pm})$ is isomorphic to $(g_1,J_{1,\pm})$ near $P_2$,
so that $\tilde g_2=g_2$ away from $P_2$, and so that
$||\tilde g_2-g_2||^1<\epsilon$.
\end{theorem}

\begin{proof}By Proposition~\ref{prop-2.1}~(4), we may choose coordinate systems $\vec x$ on $M_1$ and $\vec y$ on $M_2$
so that $g_1(P_1)=g_2(P_2)$ under the identification
described in Equation~(\ref{eqn-2.e}) or Equation~(\ref{eqn-2.d}), so that
$J_{1,\pm}(P_1)=J_{2,\pm}(P_2)$ under the identification
described in Equation~(\ref{eqn-2.f}), and so that the first derivatives of the metrics vanish at $P_1$ and at $P_2$.
(This is, of course, exactly the step which
fails in examining the (para)-Hermitian setting and is why
we only obtain $C^0$ norm estimates and not $C^1$ norm estimates in those cases.)
We then have for $r$ sufficiently small that
$$||\Omega_{\pm,g_1}-\Omega_{\pm,g_2}||_{3r}\le Cr^2\text{ and }
||\Omega_{\pm,g_1}-\Omega_{\pm,g_2}||_{3r}^1\le Cr\,.$$

Unfortunately, the convex combination of (para)-K\"ahler
metrics is no longer a (para)-K\"ahler metric so we must proceed a bit differently.
We shall exploit a correspondence between $S_\mp^2(M)$ and $\Lambda_\mp^2(M)$.
 If $h\in C^\infty(S_\mp^2(M))$, then we may define $\omega_{\pm,h}\in C^\infty(\Lambda_\mp(M)$
 by setting $\omega_{\pm,h}(x,y):=h(x,J_\pm y)$;
$h$ may be recovered from $\omega_{\pm,h}$ since $h(x,y)=\pm\omega_h(x,J_\pm y)$.
Instead of finding $\tilde g_2$ directly, we will find
$\tilde\omega_\pm\in C^\infty(\Lambda_\mp^2(B_{3r}))$ so that
\begin{equation}\begin{array}{ll}\label{eqn-2.h}
\tilde\omega_\pm=\Omega_{\pm,g_1}\text{ on }B_{r},&
\tilde\omega_\pm=\Omega_{\pm,g_2}\text{ on }B_{2r}^c,\\
||\tilde\omega_\pm-\Omega_{\pm,g_2}||_{3r}^1<\epsilon,&
d\tilde\omega_\pm=0\,.\vphantom{\vrule height 11pt}
\end{array}\end{equation}
We will then set $\tilde g(x,y)=\pm\tilde\omega_\pm(x,J_\pm y)$ and conclude
that $(\tilde g_2,J_{2,\pm})$ is a (para)-Hermitian
manifold with:
$$
\tilde g_2=g_1\text{ on }B_{r},\quad
\tilde g_2=g_2\text{ on }B_{2r}^c,\quad
||\tilde g_2-g_2||_{3r}^1<\epsilon\,.
$$
The estimate $||\tilde g_2-g_2||_{3r}^1<\epsilon$ shows that $\tilde g_2$ is non-degenerate. The identity
$d\Omega_{\pm,\tilde g_2}=d\tilde\omega_\pm=0$ shows that
$(\tilde g_2,J_{2,\pm})$ is
(para)-K\"ahler structure on $M_2$
and completes the proof. We now proceed with the construction of $\tilde\omega_\pm$. Set:
$$\Delta_\pm:=\Omega_{\pm,g_1}-\Omega_{\pm,g_2}\text{ on }B_{3r}\,.$$
We then have that
$$
d\Delta_\pm=d\Omega_{\pm,g_1}-d\Omega_{\pm,g_2}=0,\quad
||\Delta_\pm||_{3r}^1\le Cr,\quad\text{ and }\quad
||\Delta_\pm||_{3r}\le Cr^2\,.
$$

We first work in the complex setting. Let $\partial$ and $\bar\partial$ be the {\it Dolbeault operators}.
The Dolbeault cohomology of $B_{3r}$ vanishes
so we may find $f_-\in C^\infty(B_{3r})$ so that
$$\Delta_-=\sqrt{-1}\partial\bar\partial f_-\,.$$
The analogue of $\partial$ and $\bar\partial$ in para-complex
geometry are the operators $d_\pm$. But again, we can find a
function $f_+\in C^\infty(B_{3r})$ so that
$$\Delta_+=d_+d_-f_+\,.$$
We refer to \cite{CMMS04,H11,KN63} for further details concerning the {\it (para)-K\"ahler
potential} $f_\pm$. To have a uniform
notation, we set
 $$A_+=d_+,\quad A_-=\sqrt{-1}\partial,\quad B_+=d_-,\quad B_-=\bar\partial\,.$$
Expand $f_\pm$ and $\Delta_\pm$ in Taylor series about $0$ in the form
$$
\textstyle\Delta_\pm\sim\sum_\nu\Delta_{\pm,\nu}\quad\text{ and }\quad f_\pm\sim\sum_\nu f_{\pm,\nu}
$$
where $\Delta_{\pm,\nu}$ and $f_{\pm,\nu}$ are polynomials
which are homogeneous of
degree $\nu$. Since $g_1(0)=g_2(0)$ and $dg_1(0)=dg_2(0)=0$, we have
$\Delta_{\pm,\nu}=0$ for $\nu=1,2$. Equating terms of suitable homogeneity in the expansions
$\sum_\mu A_\pm B_\pm\ f_{\pm,\mu}\sim\sum_\nu\Delta_{\pm,\nu}$ yields:
\begin{eqnarray*}
&&A_\pm B_\pm f_{\pm,0}=0,\quad A_\pm B_\pm f_{\pm,1}=0,\quad
A_\pm B_\pm f_{\pm,2}=\Delta_{\pm,0}=0,\\
&&A_\pm B_\pm f_{\pm,3}=\Delta_{\pm,1}=0\,.
\end{eqnarray*}
Consequently we may assume that
 $f_{\pm,0}=f_{\pm,1}=f_{\pm,2}=f_{\pm,3}=0$. This implies that there exists a constant $C$ so that if $0<t<1$, then we have:
$$
 ||\partial_{x_j}\partial_{x_k} f_\pm||_{3rt}<C\cdot t^2\text{ and }
   ||\partial_{x_i}\partial_{x_j}\partial_{x_k}f_\pm||_{3rt}<C_0t\,.\vphantom{\vrule height 11pt}
$$
The same technique used to establish Theorem~\ref{thm-2.2} permits us to estimate
\begin{equation}\label{eqn-2.i}
||A_\pm B_\pm\phi_{rt}f_\pm||\le Ct^2\text{ and }||\partial_iA_\pm B_\pm\phi_{rt}f_\pm||\le Ct\,.
\end{equation}
Set
$$\tilde\omega_\pm:=\Omega_{\pm,g_2}+A_\pm B_\pm \phi_{rt}f_{\pm}\,.$$
We must verify that the relations of Equation~(\ref{eqn-2.h})
are satisfied where we replace $r$ by $rt$.
On $B_{rt}$, we have $\phi_{rt}=1$. Consequently
\begin{eqnarray*}
&&\left.\tilde\omega_\pm\right|_{B_{rt}}=
\left.\{\Omega_{\pm,g_2}+A_\pm B_\pm\phi_{rt}f_\pm\}\right|_{B_{rt}}
=\{\Omega_{\pm,g_2}+A_\pm B_\pm f_\pm\}|_{B_{rt}}\\
&=&\left\{\Omega_{\pm,g_2}+\Delta_\pm\}\right|_{B_{rt}}=
\left\{\Omega_{\pm,g_2}+\Omega_{\pm,g_1}-\Omega_{\pm,g_2}\}\right|_{B_{rt}}
=\left.\Omega_{\pm,g_1}\right|_{B_{rt}}\,.
\end{eqnarray*}
On $B_{2rt}^c$ we have $\phi_{rt}=0$. Consequently
$$
\left.\tilde\omega_\pm\right|_{B_{2rt}^c}=
\left.\{\Omega_{\pm,g_2}+A_\pm B_\pm\phi_{rt}f_\pm\}\right|_{B_{2rt}^c}
=\left\{\Omega_{\pm,g_2}+0\}\right|_{B_{2rt}^c}\,.
$$
The identity $\tilde\omega_{\pm}=0$ follows from the fact that $d\bar\partial\partial=0$ in
the complex setting and $dd_+d_-=0$ in the para-complex setting. The desired estimate for $||\tilde g_2-g_2||^1$ follows from Equation~(\ref{eqn-2.i}).
\end{proof}

\subsection{Weyl structures}
We refer to \cite{BGN12,GNS11} for background information concerning Weyl
geometry.
A {\it Weyl structure} on $M$ is a pair $(M,\nabla)$ where $g$ is a pseudo-Riemannian
 metric on $M$ and where $\nabla$ is a torsion free
connection on $M$ so that $\nabla g=-2\omega\otimes g$ for some smooth $1$-form $\omega$.
These structures were
introduced by Weyl~\cite{W22} in an attempt to unify gravity
with electromagnetism. Although this approach failed for physical reasons,
these geometries are still studied for their intrinsic interest.
This is a conformal theory. If $\tilde g=e^{2f}g$ is a conformally equivalent metric,
then $(M,\tilde g,\nabla)$ also is a Weyl
structure where the associated $1$-form is given by setting
$\tilde\omega=\omega+2df$. The triple $(M,g,\nabla)$ is called a {\it Weyl manifold}
and $(g,\nabla)$ is said to give $M$ a {\it Weyl structure}.

\begin{theorem}
Let $(g_1,\nabla_1)$ be the germ of an  Weyl structure of signature $(p,q)$ at $P_1\in M_1$
 and let $(g_2,\nabla_2)$ be a Weyl structure of signature $(p,q)$ on $M_2$. There exists
a Weyl structure $(\tilde g_2,\nabla_2)$ on $M_2$ so
that $(\tilde g_2,\tilde\nabla_2)$ is equal to $(g_2,\nabla_2)$ away from $P_2$ and
so that $(\tilde g_2,\tilde\nabla_2)$ is isomorphic to $(g_1,\nabla_1)$ near $P_2$.
\end{theorem}

\begin{remark}\rm
It is not possible to choose $(\tilde g_2,\tilde\nabla_2)$ so that both of the following estimates hold:
$$||g_2-\tilde g_2||_{3r}<\epsilon\text{ and }
||\nabla_2-\tilde\nabla_2||_{3r}<\epsilon\,.$$
The associated $1$ forms $\omega_i$ defined by the identity $\nabla_ig_i=\omega_i\otimes g_i$
are invariants of the theory.
If we assume $g_1(P_1)=g_2(P_2)$, then
$$g_2(P_2)(\omega_1-\omega_2,\omega_1-\omega_2)=g_2(P_2)(\tilde\omega_2-\omega_2,\tilde\omega_2-\omega_2)$$ is an invariant which can in general be made arbitrarily small.
\end{remark}

\begin{proof} We replace $B_{3r}$ by $B_{5r}$ and use the construction used to prove
Theorem~\ref{thm-2.2}
to choose a pseudo-Riemannian metric $g_2$ on $M_2$ and so that $\tilde g_2=g_1$ on $B_{2r}$ so
that $\tilde g_2=g_2$ on $B_{3r}^c$. We generalize the argument of Section~\ref{sect-2.1}
to this setting. Let $\{\psi_1,\psi_2,\psi_3\}$ be
a partition of unity on $M$ which is subordinate to the cover $\{B_{2r},B_{4r}\cap B_r^c,B_{3r}^c\}$.
Let $\nabla_{\tilde g_2}$ be the Levi-Civita connection of $\tilde g_2$.
Form:
$$
\tilde\nabla_2=\psi_1\nabla_1+\psi_2\nabla_{\tilde g_2}+\psi_3\nabla_2\,.
$$
The convex combination of Weyl connections
for a metric $g$ is again a Weyl connection for that metric.
We argue:
\begin{enumerate}
\item On $B_{2r}$, we have $\psi_3=0$ and $\tilde g_2=g_1$ so
$\tilde\nabla_2=\psi_1\nabla_1+\psi_2\nabla_{\tilde g_2}$
is  a Weyl connection for $\tilde g_2=g_1$.
\item On $B_{2r}^c\cap B_{3r}$, we have $\psi_1=0$, $\psi_3=0$, and $\psi_2=1$ so
$\tilde\nabla_2=\nabla_{\tilde g_2}$ is a Weyl connection for $\tilde g_2$.
\item On $B_{3r}^c$, we have $\tilde g_2=g_2$ and $\psi_1=0$ so
$\tilde\nabla_2=\psi_2\nabla_{\tilde g_2}+\psi_3\nabla_2$ is a Weyl connection for $\tilde g_2=g_2$.
\end{enumerate}
Note that we do not have the $\epsilon$ condition
as $\tilde\nabla_2$ need not be close to $\nabla_2$ at $P_2$.
\end{proof}

\subsection{(para)-K\"ahler--Weyl structures}
A quadruple $(M,g,J_\pm,\nabla)$ is said to be a
{\it (para)-K\"ahler--Weyl structure} if $(M,g,J_\pm)$ is a (para)-Hermitian manifold,
if $(M,g,\nabla)$ is a Weyl structure, and if $\nabla J_\pm=0$; the triple $(g,J_\pm,\nabla)$
is said to give $M$ a {\it (para)-K\"ahler--Weyl structure}.

\begin{theorem}
Let $(g_1,J_{1,\pm},\nabla_1)$ be the germ
of a (para)-K\"ahler--Weyl structure of signature $(p,q)$ at $P_1\in M_1$
 and let $(g_2,J_{2,\pm},\nabla_2)$ be a
 (para)-K\"ahler--Weyl structure of signature $(p,q)$ on $M_2$. There exists
$(\tilde g_2,\tilde\nabla_2)$ so that
$(\tilde g_2,J_{2,\pm},\tilde\nabla_2)$ is
(para)-K\"ahler--Weyl structure  on $M_2$ with $(\tilde g_2,\tilde\nabla_2)$ equal to
$(g_2,\nabla_2)$ away from $P_2$, and so that
 $(\tilde g_2,\tilde J_{2,\pm},\tilde\nabla_2)$ isomorphic to
$(g_1,J_{1,\pm},\nabla_1)$ near $P_2$.\end{theorem}

\begin{proof} We apply the analysis of \cite{GN12} and divide the analysis into 2 cases..
We first suppose that $m\ge6$. Let $i=1,2$. Then near $P_i$ we have that
$g_i=e^{2\Phi_i}h_i$ is conformally equivalent to a K\"ahler
metric $h_i$ on $(B_{3r},J_{i,\pm})$ with $\nabla_i=\nabla^{h_i}$ being the
Levi-Civita connection of $h_i$. We apply argument used to prove Theorem~\ref{thm-2.7}
to construct a K\"ahler metric $\tilde h_2$ on $B_{3r}$ so that
$\tilde h_2=h_1$ on $B_r$ and so that $\tilde h_2=h_2$ on
 $B_{2r}^c$. We set $\tilde\nabla_2=\nabla^{\tilde h_2}$.
Then $\tilde\nabla_2=\nabla_1$ on $B_r$
and $\tilde\nabla_2=\nabla_2$ on $B_{2r}^c$; consequently,
$\tilde\nabla_2$ is globally defined and we have that $\tilde\nabla_2J_{2,\pm}=0$.
However, we can no longer conclude
that $||\tilde\nabla_2-\nabla_2||<\epsilon$ since
we have not controlled the first derivatives of $\tilde h_2-h_2$ on $B_{3r}$. We smooth
out the conformal factor setting
$$\Phi:=\phi_r\Phi_1+(1-\phi_r)\Phi_2\text{ and }
\tilde g_2=e^{2\Phi}\tilde h_2\,.$$
We then have $\tilde g_2=g_1$ on $B_r$ and
$\tilde g_2=g_2$ on $B_{2r}^c$.
As $(\tilde g_2,J_{2,\pm},\tilde\nabla_2)=(g_2,J_{2,\pm},\nabla_2)$
on $B_{2r}^c$, $(\tilde g_2,J_{2,\pm},\tilde\nabla_2)$ is a
(para)-K\"ahler--Weyl structure on $B_{2r}^c$. Since $\tilde\nabla_2$ is the
Levi-Civita connection of the K\"ahler metric $\tilde h_2$ on $B_{2r}$,
$(\tilde h_2,J_{2,\pm},\tilde\nabla_2)$ is a (para)-K\"ahler--Weyl structure on $B_{2r}$.
Since $\tilde g_2$ is conformally equivalent to $\tilde h_2$, it follows that
$(\tilde g_2,J_{2,\pm},\tilde\nabla_2)$ is (para)-K\"ahler--Weyl structure on $B_{2r}$ as well. This completes the proof if $m\ge6$.

In dimension $m=4$, every (para)-Hermitian manifold $(M,g,J_\pm)$ admits a unique
(para)-K\"ahler structure where the associated 1-form is given by the anti-Lee form
$\pm\frac12J^*\delta\Omega_\pm$.
The desired result now follows from Theorem~\ref{thm-2.6}; again,
we lose control on $||\tilde\nabla_2-\nabla_2||$ since we have
not controlled the norms of the first derivatives of $\tilde g_2-g_2$.
\end{proof}

\subsection{Geodesic completeness}\label{sect-2.10}
In Riemannian geometry, compact manifolds are necessarily geodesically complete. This is not
the case in the higher signature setting. For example, Misner \cite{M63} showed that the metric
$$ds^2:=\cos(x)(dy\circ dy-dx\circ dx)+2\sin(x)dx\circ dy$$
on the 2-dimensional torus is geodesically incomplete. Similarly, Meneghini \cite{M04}
showed the metric
$$ds^2:=du\circ dv/(u^2+v^2)\text{ on }\mathbb{R}^2-\{(0,0)\}$$
is geodesically incomplete; this may be compactified by identifying $(u,v)$ with $(2u,2v)$ to
obtain a geodesically incomplete metric on the 2-torus. Finally, taking
$\Gamma_{11}{}^1=1$ on the circle yields a geodesically incomplete affine manifold. Thus
 the following transplantation result does not follow from Theorem~\ref{thm-1.1}:

\begin{theorem}\label{thm-2.10}
Let $\epsilon>0$ be given.
\begin{enumerate}
\item Let $\nabla_1$ be the germ of a torsion free connection at $P_1\in M_1$ and let $\nabla_2$ be the usual flat connection on $\mathbb{R}^m$.
There exists a torsion free connection $\tilde\nabla_2$ on
$\mathbb{R}^m$ which is isomorphic to $\nabla_1$ on $B_\epsilon$,
which agrees with $\nabla_2$ on $B_{2\epsilon}^c$, which
satisfies the estimate $||\tilde\nabla_2-\nabla_2||<\epsilon$, and which is geodesically complete.
\item Let $g_1$ be the germ of a pseudo-Riemannian metric of
signature $(p,q)$ at $P_1\in M_1$ and let $g_2$ be a flat
pseudo-Riemannian metric of signature $(p,q)$ on $\mathbb{R}^m$.
There exists a pseudo-Riemannian metric $\tilde g_2$
on $\mathbb{R}^m$ which is isomorphic to $g_1$ on $B_\epsilon$,
which agrees with $g_2$ on $B_{2\epsilon}^c$,
which satisfies the estimate $||\tilde g_2-g_2||^1<\epsilon$, and which is geodesically complete.
\item Let $(g_1,J_{1,\pm})$ be the germ of a (para)-K\"ahler structure near $P_1\in M_1$
of signature $(p,q)$. Let $(g_2,J_{2,\pm})$
be a flat (para)-K\"ahler structure on $\mathbb{R}^m$ of signature $(p,q)$.
There exists pseudo-Riemannian metric $\tilde g_2$
on $\mathbb{R}^m$ so that $(\tilde g_2,J_{2,\pm})$ is a K\"ahler structure, so that
 $(g_1,J_{1,\pm})$ is isomorphic to $(\tilde g_2,J_{2,\pm})$ on
$B_\epsilon$, so that $\tilde g_2$ agrees with $g_2$
on $B_{2\epsilon}^c$, so that $\tilde g_2$ satisfies the estimate
 $||\tilde g_2-g_2||^1<\epsilon$, and so that $\tilde g_2$ is geodesically complete.
\end{enumerate}
\end{theorem}

Theorem~\ref{thm-2.10} will follow from the $C^0$ estimate of Theorem~\ref{thm-2.1},
from the $C^1$ estimate of Theorem~\ref{thm-2.2}, from the $C^1$ estimate of
 Theorem~\ref{thm-2.7}, and from the
following result:

\begin{lemma} Let $\epsilon>0$ be given with $\epsilon<\epsilon(m)$ sufficiently small. Let $\nabla_2$ be the flat connection
on $\mathbb{R}^m$. If $\tilde\nabla_2$ is a torsion free connection on $\mathbb{R}^m$ with $\tilde\nabla_2=\nabla_2$ on
$B_\epsilon^c$ and $||\tilde\nabla_2-\nabla_2||<\epsilon$, then $\tilde\nabla_2$ is geodesically complete.
\end{lemma}

\begin{proof} Let $\gamma$  be a geodesic for $\tilde\nabla_2$. If $\gamma$ never enters $B_\epsilon$, then $\gamma$ is a straight line and continues for infinite time. Suppose $\gamma(0)\in B_\epsilon$;
we can assume without loss of generality that $|\dot\gamma(0)|=1$ relative to the usual Euclidean metric on $\mathbb{R}^m$. We examine
the situation for $t>0$ as the situation for $t<0$ is analogous. Assume $\gamma$ does not extend for infinite time and let $[0,T\vphantom{(})$ be a
maximal domain. If $\gamma$ exits from $B_\epsilon$, then it continues as a straight line for infinite time and consequently $\gamma$
is trapped inside $B_\epsilon$. Let $\dot\gamma:=\partial_t\gamma$ and $\ddot\gamma:=\partial_t^2\gamma$. By compactness, there is a uniform $\delta$ so if
$|\dot\gamma(t)|\le 2$, then $\gamma$ extends to $[t,t+\delta\vphantom{(})$.
Consequently $|\dot\gamma(t_0)|>2$ for some $t_0\in[0,T\vphantom{(})$. Choose $t_0$ minimal
so $|\dot\gamma(t_0)|=2$ and thus $|\dot\gamma|\le 2$ on $[0,t_0]$.
The geodesic equation $\ddot\gamma+\Gamma(\dot\gamma,\dot\gamma)=0$ yields the estimate
$$|\ddot\gamma(t)|\le C_m|\Gamma|\cdot|\dot\gamma(t)|^2\le C_m\cdot\epsilon\cdot 4\,.$$
Consequently, if $\epsilon$ is small, we may conclude that $|\ddot\gamma|$ is small. Since $|\dot\gamma|\ge2$ and since $\dot\gamma$
is changing rather slowly, this implies $t_0$ must be relative large; in particular, for very small $\epsilon$, we must have $t_0\ge3$.
Again, as $\dot\gamma$ is changing slowly, we have $\dot\gamma(t)-\dot\gamma(0)$ must be small. In particular,
$$(\dot\gamma(t),\dot\gamma(0))\ge\textstyle\frac12(\dot\gamma(0),\dot\gamma(0))\ge\frac12\,.$$
Consequently $(\gamma(t),\dot\gamma(0))\ge\frac12t$. Taking $t_0=3$ then implies $|\gamma(t_0)|\ge\frac32$ and hence,
contrary to our assumption, $\gamma$
has exited from $B_\epsilon$. This contradiction establishes the Lemma.
\end{proof}\black

\section{Universal curvature identities}\label{sect-3}
In this section, we use Theorem~\ref{thm-1.2} to answer certain questions that were
raised previously in \cite{EPS4, LPS}.

\subsection{The identity of Berger}
The generalized Gauss-Bonnet theorem expresses a topological invariant (the Euler characteristic) as an
integral of an expression in curvature (the Euler form) in the category of compact Riemannian manifolds.
Berger \cite{Beg} used the associated Euler-Lagrange equations to derive a curvature identity on any 4-dimensional
compact oriented Riemannian manifold. Let $\check R$, $\check\rho$, and $L\rho$ be the symmetric $2$-tensors
whose components are given by:
$$
\check{R}_{ij}=\sum_{a,b,c}R_{abci}{R^{abc}}_{j},\quad
\check{\rho}_{ij}=\sum_a\rho_{ai}\rho^{a}_{~j},\quad
(L\rho)_{ij}=2\sum_{a,b}R_{iabj}\rho^{ab}.
$$
The identity of Berger is quadratic in the curvature; it may be stated as follows:
   \begin{equation}\label{eqn-3.a}
    \textstyle\frac{1}{4}(|R|^2-4|\rho|^2+\tau^2)g-\check{R}+2\check{\rho}+L\rho-\tau\rho=0.
    \end{equation}
Euh, Park, and Sekigawa \cite{EPSx,EPS2,EPS3} gave a direct proof that Equation~(\ref{eqn-3.a}) holds more generally
in the non-compact setting and provided some applications of the identity. There is a higher dimensional
generalization found by Kuz'mina \cite{K} and subsequently established using different
method by Labbi \cite{La1, La2, La}.
Gilkey, Park, and Sekigawa \cite{GPS} gave a different proof of this higher dimensional generalization using heat
trace method based on Weyl's invariance theory \cite{W23}. We also refer to related work of Euh, Jeong, and Park \cite{EJP}.
It follows from Theorem~\ref{thm-1.2} that any such identity which holds in the compact setting automatically extends
to the non-compact setting  thus answering directly a question
originally posed in \cite{EPS4, LPS}.

\subsection{Almost Hermitian geometry} Work of Gray \cite{Gr} and of Tricerri and Vanhecke \cite{TV81}
establishes several curvature identities in the context of almost
Hermitian geometry. The defining equations of Ricci $*$-tensor and
the $*$-scalar curvature are examples of such identities. Blair
\cite{Bl} discussed a similar problem for almost contact metric
manifolds. Gray and Hervella \cite{GH} examined identities related
to the K\"ahler form {{$\Omega$}}. We shall present an
application of Theorem~\ref{thm-1.2} in this setting. Let
$\rho^*=(\rho^*_{ij})$ and $\tau^*=g^{ij}\rho_{ij}^*$ be the Ricci
$*$-tensor and the $*$-scalar curvature of $M$, respectively. The
following identity holds on any 4-dimensional almost Hermitian
manifold $(M,g,J)$ \cite{Gr, GH}:
    \begin{equation}\label{eqn-3.b}
\textstyle    \rho^*_{ij}+\rho^*_{ji}-\rho_{ij}-J^{~a}_{i}J^{~b}_{j}\rho_{ab}-\frac{\tau^*-\tau}{2}g_{ij}=0.
    \end{equation}
It is easily checked that the equality in Equation~\eqref{eqn-3.b}
is trivial in the case where $M$ is K\"ahler since the following
identity holds on any K\"ahler manifold:
$$R_{ijkl}=R_{ijab}{J_k}^a{J_l}^b\,.$$
Let $M=(M,g,J)$  be a 4-dimensional compact
almost Hermitian manifold; we replace $J_-$ by $J$ to avoid unnecessary notational complexities.
Let $\chi(M)$, $p_1(M)$ and $c_1^2(M)$ be the Euler number of $M$, the
first Pontrjagin number of $M$, and the evaluation of the square of the first Chern form on $M$, respectively.
Then Wu's Theorem \cite{Wu} yields the relationship:
\begin{equation}\label{eqn-3.c}
        c_1^2(M)=2\chi(M)+p_1(M)\,.
\end{equation}

Euh, Park, and Sekigawa \cite{EPS4} derived the following
curvature identities on $M$ from the integral formula for the first
Pontrjagin number using Novikov's theorem \cite{N} provided that the underlying manifold
$M$ in question was compact:
\begin{equation}\label{eqn-3.d}
\begin{aligned}
T_{ij} = T_{ab}{J_i}^a{J_j}^b,\quad S_{ij} = S_{ab}{J_i}^a{J_j}^b,\quad
T_{ib}{J_j}^b + S_{ij} = 0\,.
\end{aligned}
\end{equation}
where
\medbreak\qquad
$T_{ij}:=\frac{1}{2}\big(T'_{ij}+T'_{ji}\big)$,\quad
$S_{ij}:=\frac{1}{2}\big(S'_{ij}-S'_{ji}\big)$,
\smallbreak\qquad
$T'_{ij}:=2\big({{\rho^*}^a}_{j}\rho^*_{ai}-\rho^*_{ja}{{\rho^*}_i}^a+{\rho^*}^{ab}R_{auvi}{J_b}^{u}{J_j}^{v}$
\smallbreak\qquad\quad
$+2\nabla_b\nabla_a({\rho^*}^{ac}{J_i}^b{J_{cj}})
+\frac{1}{2}{\rho^*}^{ab}\rho^*_{ab}g_{ij}\big)$
\smallbreak\qquad\quad
$4{J_j}^aJ^{ub}{R_{abk}}^l{R_{iul}}^k+4\nabla_b\nabla_a({J_u}^aJ_{vi}{{R^{uv}}_{j}}^b)$
\smallbreak\qquad\quad
$+\frac{1}{2}(J^{ua}J^{vb}{R_{abk}}^l{R_{uvl}}^k)g_{{ij}}$,
\medbreak\qquad
$S'_{ij}:=2{\rho^*}^{ab}(2J_{bj}\rho^*_{ai}-{J_b}^c{J_i}^u{J_j}^vR_{acuv})
-2(J^{ua}{R_{jak}}^l{R_{iul}}^k)$.
\medbreak
Since $\chi(M)$ and $p_1(M)$ are topological invariants, Equation \eqref{eqn-3.c} shows that
$c_1^2(M)$ also is a topological invariant. Using this observation, Lee, Park, and Sekigawa
\cite{LPS} derived the following curvature identities on $M$ from the
integral formula for $c_1^2(M)$ in terms of curvature
provided that the underlying manifold $M$ in question was compact:
\begin{equation}\label{eqn-3.e}
U_{ij} = U_{ab}{J_i}^a{J_j}^b,\quad V_{ij} = V_{ab}{J_i}^a{J_j}^b,\quad
U_{ib}{J_j}^b + V_{ij} = 0
\end{equation}
where
\medbreak\qquad
$U^{ij}:=\frac{1}{2}({U'}^{ij}+{U'}^{ji})$,
${{V_{ij}:=g_{ip}g_{jq}V^{pq}}}$,
$V^{pq}=\frac{1}{2}({V'}^{pq}{{-}}{V'}^{qp})$,
\medbreak\qquad
${U'}^{ij}:=J^{iv}{\rho^*_v}^w(\nabla_w{J_a}^b)(\nabla^j{J_u}^a){J_b}^u
+J^{wi}{\rho^*}^{jc}(\nabla_c{J_a}^b)(\nabla_w{J_u}^a){J_b}^u$
\smallbreak\qquad\quad
$+4\nabla_u(J^{cv}{\rho^*_v}^i(\nabla_cJ^{ju}))+\frac{1}{2}J^{cv}{R_{vws}}^iJ^{dw}J^{sj}(\nabla_d{J_a}^b)(\nabla_c{J_u}^a){J_b}^u$
\smallbreak\qquad\quad
$-\nabla_s\nabla_c(J^{dj}J^{si}{J_k}^c(\nabla^k{J_a}^b)(\nabla_d{J_u}^a){J_b}^u){{-3J^{cv}{\rho^*_v}^d(\nabla_d{J_a}^i)(\nabla_cJ^{jw}){J_w}^a}}$
\smallbreak\qquad\quad
$-\frac{1}{2}J^{cv}{\rho^*_v}^d(\nabla_d{J_a}^b)(\nabla_c{J_u}^a){J_b}^ug^{ij}-\frac{1}{2}{\rho^*}^{ij}(\nabla_c{J_a}^b)(\nabla_u{J_v}^a){J_b}^vJ^{cu}$
\smallbreak\qquad\quad
$+\frac{1}{2}\nabla_l\nabla_k(J^{ik}J^{jl}(\nabla_c{J_a}^b)(\nabla_u{J_v}^a){J_b}^vJ^{cu})+2\nabla_c(\tau^*(\nabla_uJ^{jc})J^{iu})$
\smallbreak\qquad\quad
${{+\frac{3}{2}\tau^*(\nabla_cJ^{jb})(\nabla_u{J_v}^i){J_b}^vJ^{cu}}}+\tau^*(\nabla^j{J_a}^b)(\nabla_u{J_v}^a){J_b}^vJ^{iu}$
\smallbreak\qquad\quad
$-\frac{1}{4}\tau^*(\nabla_c{J_a}^b)(\nabla_u{J_v}^a){J_b}^vJ^{cu}g^{ij}+{{2\rho^*_{lk}{{R^k}_{ab}}^iJ^{la}J^{bj}}}$
\smallbreak\qquad\quad
${{+4\nabla_a\nabla^l(\rho^*_{kl}J^{ai}J^{kj})}} -{\rho^*}^{ab}{{\rho^*_{ba}}}g^{ij}+2\tau^*{\rho^*}^{ij}$
\smallbreak\qquad\quad
$-2\nabla_b\nabla_a(\tau^*J^{ia}J^{jb})+\frac{1}{2}(\tau^*)^2g^{ij}$ and
\medbreak\qquad
${V'}^{pq}:=-J^{{{j}}v}J^{iq}{\rho^*_v}^p(\nabla_i{J_a}^b)(\nabla_j{J_u}^a){J_b}^u$
\smallbreak\qquad\quad
$+\frac{1}{2}J^{jv}J^{iw}J^{pc}J^{qd}R_{vwcd}(\nabla_i{J_a}^b)(\nabla_j{J_u}^a){J_b}^u$
\smallbreak\qquad\quad
$-2\nabla_i(J^{jv}{\rho^*_v}^i(\nabla_j{J_w}^p)J^{qw})+{\rho^*}^{qi}(\nabla_i{J_a}^b)(\nabla^p{J_u}^a){J_b}^u$
\smallbreak\qquad\quad
$+J^{{{j}}v}{\rho^*_v}^i(\nabla_i{J_a}^p)(\nabla_jJ^{qa})$
\smallbreak\qquad\quad
$+J^{ip}{\rho^*_i}^q(\nabla_c{J_a}^b)(\nabla_u{J_v}^a){J_b}^vJ^{cu}+\frac{1}{2}\tau^*(\nabla_c{J_a}^p)(\nabla_uJ^{qa})J^{cu}$
\smallbreak\qquad\quad
$+\frac{1}{2}\tau^*(\nabla^p{J_a}^b)(\nabla^q{J_v}^a){J_b}^v-\nabla_u(\tau^*(\nabla_cJ^{qb}){J_b}^pJ^{cu})-4J^{lq}\rho^*_{lk}{\rho^*}^{kp}$
\smallbreak\qquad\quad
$+{{2\rho^*_{lk}J^{lc}J^{pa}J^{qb}{R^k}_{cab}}}-4\tau^*J^{ap}{\rho^*_a}^q$.
\medbreak\noindent
Then, from the Theorem~\ref{thm-1.2}, it follows that Equations \eqref{eqn-3.d}
 and {{\eqref{eqn-3.e}}} also hold on any
4-dimensional almost Hermitian manifold which is not necessarily compact. This answers affirmatively
questions posed in \cite{EPS4,LPS}. In particular, if $M=(M,g,J)$  is a 4-dimensional
K\"ahler manifold one has the following curvature identities \cite{EPS4, LPS}:
\begin{eqnarray*}
&&\textstyle 2\check{\rho} - \tau{\rho}-\frac{1}{2}\Big(|\rho|^2-\frac{\tau^2}{2}\Big)g=0,\text{ and }
\textstyle 2L\rho-
8\check{\rho}+2\tau\rho+\frac{1}{2}\Big(2|\rho|^2-{{{\tau^2}}}\Big)g=0.
\end{eqnarray*}

\section{Realizing curvature tensors}\label{sect-4}
In this section, we present two examples to illustrate the use of Theorem~\ref{thm-1.3} to solve problems involving geometric realizability.
In Section~\ref{sect-4.1}, we study pseudo-Riemannian curvature models and in
Section~\ref{sect-4.2}, we study para-K\"ahler curvature models. We emphasize that
each of the contexts we have examined have results of a similar nature but we have limited ourselves to these two examples in the interests of
brevity; in particular, the K\"ahler setting is completely analogous.

\subsection{Riemannian curvature models}\label{sect-4.1}
The Riemann curvature tensor satisfies the following symmetries
\begin{eqnarray}
&&A(x,y,z,w)+A(y,x,z,w)=0,\nonumber\\
&&A(x,y,z,w)+A(y,z,x,w)+A(z,x,y,w)=0,\label{eqn-4.a}\\
&&A(x,y,z,w)+A(x,y,w,z)=0.\nonumber
\end{eqnarray}
We say that a triple $(V,\langle\cdot,\cdot\rangle,A)$ is a {\it pseudo-Riemannian curvature model} of signature $(p,q)$ if $(V,\langle\cdot,\cdot\rangle)$
is an inner product space of signature $(p,q)$ and if $A\in\otimes^4V^*$ satisfies the universal curvature symmetries given above. We say
that two such models $(V_i,\langle\cdot,\cdot\rangle_i,A_i)$ are {\it isomorphic} if there is a linear map $\phi$ from $V_1$ to $V_2$ so that $\phi^*A_2=A_1$
and so that $\phi^*\langle\cdot,\cdot\rangle_2=\langle\cdot,\cdot\rangle_1$. The following theorem shows that the universal symmetries given in Equation~(\ref{eqn-4.a}) generate all the curvature symmetries. There are no additional ``hidden" symmetries and
any other universal identity is an algebraic consequence of these identities:

\begin{theorem}\label{thm-4.1}
Let $\mathfrak{M}=(V,\langle\cdot,\cdot\rangle,A)$ be a pseudo-Riemannian curvature model of signature $(p,q)$. Let
$P$ be a point of a pseudo-Riemannian manifold $(M,g)$ of signature $(p,q)$.
Let $\varepsilon>0$ be given. There exists a pseudo-Riemannian
metric $\tilde g$ on $M$ so that $g=\tilde g$ away from $P$ and so that $(T_PM,g_P,R_P)$ is isomorphic to $\mathfrak{M}$.
\end{theorem}

\begin{proof}
Fix a basis $\{e_i\}$ for $V$. Let
$$\varepsilon_{ij}:=\langle e_i,e_j\rangle\text{ and }A_{ijkl}:=A(e_i,e_j,e_k,e_l)$$
give the structure constants of the model. Define:
$$
\textstyle g_{ik}:=\varepsilon_{ik}-\frac13A_{ijlk}x^jx^l\,.
$$
Since $\varepsilon$ is non-degenerate with signature $(p,q)$, this is the germ of a pseudo-Riemannian
metric on $(V,0)$. One uses the curvature symmetries to compute $R_{ijkl}(0)=A_{ijkl}$ (see the discussion in \cite{BGN12} for example).
We then use Theorem~\ref{thm-2.2} to transplant this structure into $M$.
\end{proof}

\subsection{Para-K\"ahler curvature models}\label{sect-4.2}
 If $(M,g,J_+)$ is a para-K\"ahler manifold, then the curvature tensor satisfies an additional symmetry
\begin{equation}\label{eqn-4.b}
R(x,y,z,w)=-R(x,y,J_+ z,J_+ w)\,.
\end{equation}
We say that a quadruple $(V,\langle\cdot,\cdot\rangle,J_+,A)$ is a
{\it para-K\"ahler curvature model} of signature $(p,q)$
if $(V,\langle\cdot,\cdot\rangle)$ is an inner product space of signature $(p,q)$
with a para-complex structure $J_+$ such that
 $J_+^*\langle\cdot,\cdot\rangle=-\langle\cdot,\cdot\rangle$, and if $A\in\otimes^4V^*$ satisfies the symmetries of Equation~(\ref{eqn-4.a})
and Equation~(\ref{eqn-4.b}). The following theorem shows that
the universal symmetries given in Equation~(\ref{eqn-4.a}) and in
Equation~(\ref{eqn-4.b}) generate all the curvature symmetries in the
para-K\"ahler setting. As in the pseudo-Riemannian
setting, there are no additional ``hidden" symmetries and
 any other universal identity is
an algebraic consequence of these identities:

\begin{theorem}\label{thm-4.2} Let $\mathfrak{M}=(V,\langle\cdot,\cdot\rangle,J_+,A)$
be a para-K\"ahler curvature model of signature $(p,q)$. Let $P$ be a point of a
para-K\"ahler manifold $(M,g,J_+)$ of signature $(p,q)$.
There exists a para-K\"ahler metric $\tilde g$ on $(M,J_+)$
so that $\tilde g=g$ away from $P$ and so that $(T_PM,g_P,J_+,R_P)$ is isomorphic to $\mathfrak{M}$.
\end{theorem}

\begin{proof} We follow the treatment in \cite{GGNV13}
to construct the germ of a suitable structure and refer to
\cite{BGM10,BGMR10} for a different treatment.
We will then use Theorem~\ref{thm-2.7} to transplant this
structure into $M$.
To simplify the notation, we assume
that $m=4$; the general case is then easily handled.
We adopt the notation of Equation~(\ref{eqn-2.g}) on the algebraic level;
$\langle\cdot,\cdot\rangle\in S_-^2(V)$. Choose a basis $\{e_i\}$ for $V$
and let $\varepsilon_{ij}=\langle e_i,e_j\rangle$.
Extend $J_+$ to an integrable para-complex structure on
$T(V)$ by identifying the tangent bundle $T(V)$ with the trivial bundle $V\times V$ over $V$:
$$\begin{array}{llll}
J_+\partial_{x_1}=\partial_{x_1},&J_+\partial_{x_2}=\partial_{x_2},&
J_+\partial_{x_3}=-\partial_{x_3},&J_+\partial_{x_4}=-\partial_{x_4},\\
J_+^* dx^1=dx^1,& J_+^* dx^2=dx^2,& J_+^* dx^3=-dx^3,& J_+^* dx^4=-dx^4\,.
\vphantom{\vrule height 12pt}\end{array}$$

Let $\theta\in S^2_-\otimes S^2$. Set:
$$g_\theta=(\varepsilon_{ij}+\theta_{ijkl}x^kx^l)dx^i\circ dx^j\,.$$
Since $g_\theta(0)=\langle\cdot,\cdot\rangle$,
$g_\theta$ is non-degenerate near $0$ and defines the germ
of a pseudo-Riemannian metric at $0\in V$. Since $\theta\in S^2_-\otimes S^2$ and
since $J_+^*\varepsilon=-\varepsilon$,
$J^*g_\theta=-g_\theta$. Thus $(g_\theta,J_+)$ defines the germ of a para-Hermitian manifold. Set:
$$
\mathcal{K}_+(\theta)(x,y,z):=2\{\theta(x,J_+ y,z,e_l)+\theta(y,J_+ z,x,e_l)
+\theta(z,J_+ x,y,e_l)\}x^l\,.
$$
It then an easy calculation to show that the exterior derivative of the K\"ahler form is given by
$$d\Omega_+^{g_\theta}(x,y,z)=2\mathcal{K}_+(\theta)(x,y,z)\,.$$
Let $\mathfrak{K}_+:=\ker(\mathcal{K}_+)\cap S_-^2\otimes S^2$.
It is then immediate that if $\theta\in\mathfrak{K}_+$, then $g_\theta$ is a para-K\"ahler metric.
Let $\mathcal{R}(\theta)$ be the curvature tensor of the Levi-Civita connection of $g_\theta$ at $0$. It is an easy calculation to verify
\begin{equation}\label{eqn-4.c}
\mathcal{R}(\theta)(x,y,z,w)=\theta(x,z,y,w)+\theta(y,w,x,z)-\theta(x,w,y,z)-\theta(y,z,x,w)\,.
\end{equation}
Let $\mathfrak{K}_{+,R}\subset\otimes^4V^*$ be the set of all tensors satisfying the symmetries of Equation~(\ref{eqn-4.a}) and
Equation~(\ref{eqn-4.b}). We will show that every
para-K\"ahler curvature model of signature $(p,q)$ can be realized by the germ of a
K\"ahler structure of signature $(p,q)$ by showing that
$$\mathcal{R}:\mathfrak{K}_+\rightarrow\mathfrak{K}_{+,R}\rightarrow0\,.$$

The Riemann curvature tensor also satisfies $R(x,y,z,w)=R(z,w,x,y)$.
Thus by Theorem~\ref{thm-4.1}, this symmetry
is an algebraic consequence of the other symmetries in Equation~(\ref{eqn-4.a});
this can also, of course, be verified combinatorially --
see, for example, the discussion in \cite{BGNS06}.
This permits us to regard $R\in S^2(\Lambda^2(V^*))$.
Set $\xi^{ijkl}:=dx^i\wedge dx^j\wedge dx^k\wedge dx^l$.
We use Equation~(\ref{eqn-4.b})
to see $\mathfrak{K}_{+,R}\subset S^2(\Lambda_+^2)$. After taking into account the Bianchi identity, we have:
\begin{eqnarray*}
&&\mathfrak{K}_{+,\mathfrak{R}}\subset\operatorname{Span}
\{\xi^{1313},\ \xi^{1414},\ \xi^{2323},\ \xi^{2424},\ \xi^{1323},\ \xi^{1424},\ \xi^{1314},\ \xi^{2324},\\
&&\phantom{\mathfrak{K}_{+,\mathfrak{R}}\subset\operatorname{Span}\ }\xi^{1324}+\xi^{1423}\}\,.
\end{eqnarray*}
\medbreak\noindent There are $9$ elements in this basis; this is in accordance with the
dimension count given by \cite{BGN12} in the para-complex setting.

We consider the following example. Let
$$
\langle\partial_{x_1},\partial_{x_3}\rangle=\langle\partial_{x_2},\partial_{x_4}\rangle=1\text{ and }
\theta=\textstyle\frac12(dx^1\circ dx^3)\otimes (dx^1\circ dx^3)\in S_-^2\otimes S^2\,;
$$
 $g_\theta$ is a para-K\"ahler metric which takes the form
$M_1\times M_2$ where $M_1$ is a
para-Riemann surface and
where $M_2$ is flat; we can also verify directly
that the para-K\"ahler form vanishes since $\mathcal{K}_+(\theta)$ is a $3$-form
which is supported on $\operatorname{Span}\{\partial_{x_1},\partial_{x_3}\}$. Furthermore, we use Equation~(\ref{eqn-4.c}) to see
\begin{equation}\label{eqn-4.d}
\mathcal{R}(\theta)=\xi^{1313}\in\operatorname{Range}(\mathcal{K}_+)\,.
\end{equation}
There is an algebra acting which will be central to our treatment. Let
$$\operatorname{End}_+=\{T\in\operatorname{End}(V):TJ_+=J_+ T\}\,.$$
The vector spaces $\mathfrak{K}_{+,R}$ and $\mathfrak{K}_+$ are
modules over $\operatorname{End}_+$ and
$\mathcal{R}:\mathfrak{K}_+\rightarrow\mathfrak{K}_{+,R}$ is an
 $\operatorname{End}_+$ module morphism.
Define $T\in\operatorname{End}_+$ by setting:
$$
T(dx^1)=dx^1+adx^2,\quad T(dx^2)=dx^2,\quad T(dx^3)=dx^3+\tilde adx^4,\quad T(dx^4)=\xi^4$$
where $a,\tilde a\in\mathbb{R}$. We apply $T$ to the
element of Equation~(\ref{eqn-4.d}) to conclude
\begin{eqnarray}
&&T(\xi^{1313})=
\xi^{1313}+2a\xi^{1323}+2\tilde a\xi^{1314}+a^2\xi^{2323}+\tilde a^2\xi^{1414}
+2a\tilde a(\xi^{1324}+\xi^{1423})\nonumber\\
&&\qquad+2a^2\tilde a\xi^{2324}+2a\tilde a^2\xi^{1424}+2a^2\tilde a^2\xi^{2424}
\in\operatorname{Range}\{\mathcal{K}_+\}\,.\label{XFE4.5c}
\end{eqnarray}
Let $\tilde a=\pm a$. Since $\xi^{1313}\in\operatorname{Range}\{\mathcal{K}_\pm\}$,  we have
$$2a\xi^{2313}\pm 2a\xi^{1314}+O(a^2)\in\operatorname{Range}\{\mathcal{K}_+\}$$ so
after dividing by $a$, we have:
$$
\xi^{2313}+O(a)\in\operatorname{Range}\{\mathcal{K}_+\}\text{ and }
\xi^{1314}+O(a)\in\operatorname{Range}\{\mathcal{K}_+\}\,.
$$
Since $\operatorname{Range}\{\mathcal{K}_+\}$ is a linear subspace of a finite dimensional vector space, it
is closed. Thus we may take the limit as $a\rightarrow0$ to conclude
$$\xi^{1323}\in\operatorname{Range}\{\mathcal{K}_+\}\text{ and }
\xi^{1314}\in\operatorname{Range}\{\mathcal{K}_+\}\,.
$$
Permuting the indices $1\leftrightarrow2$ and $3\leftrightarrow 4$ defines an element of $\operatorname{End}_\pm$ and shows
$$\xi^{2424}\in\operatorname{Range}\{\mathcal{K}_\pm\},\quad
    \xi^{1424}\in\operatorname{Range}\{\mathcal{K}_\pm\},\quad
    \xi^{2324}\in\operatorname{Range}\{\mathcal{K}_\pm\}\,.
 $$
 Since $\{a,\tilde a\}$ are arbitrary, it follows easily that
\begin{eqnarray*}
&&\{\xi^{1313},\ \xi^{1414},\ \xi^{2323},\ \xi^{2424},\ \xi^{1323},\ \xi^{1424},\ \xi^{1314},\ \xi^{2324},\\
&&\phantom\{\xi^{1324}+\xi^{1423}\}\subset\operatorname{Range}\{\mathcal{R}\}\,.
\end{eqnarray*}
This shows that it is possible to geometrically realize every
para-K\"ahler curvature model of signature $(p,q)$ as the germ
of a para-K\"ahler manifold of signature $(p,q)$.
Theorem~\ref{thm-4.2} now follows from Theorem~\ref{thm-2.7}.
\end{proof}

\section*{Acknowledgements}
\noindent This work was supported by the National Research
Foundation of Korea (NRF) grant funded by the Korea government
(MEST) (2012-0005282). It was also supported by project MTM2009-07756
(Spain). This article is dedicated to the memory of our friend and colleague E. Merino - may he rest in peace.

\end{document}